\documentclass[reqno,a4paper]{amsart}
\usepackage[
    style=numeric-comp,sorting=nty,
    sortcites=true,doi=false,url=false,
    giveninits=true]{biblatex}
    \renewbibmacro{in:}{}
\usepackage{amsmath}
\usepackage{amsthm}
\usepackage{amssymb}
\usepackage{float}
\usepackage{enumitem}
\usepackage{todonotes}
\usepackage{color}
\usepackage{graphicx}
\graphicspath{ {./Images/} }

\numberwithin{equation}{section}

\theoremstyle{plain}
\newtheorem{theorem}{Theorem}[section]
\newtheorem{lemma}[theorem]{Lemma}
\newtheorem{corollary}[theorem]{Corollary}

\newtheorem{definition}[theorem]{Definition}
\theoremstyle{remark}
\newtheorem{remark}[theorem]{Remark}

\newcommand\Pone{\textrm{P}_{\textrm{I}}}
\newcommand\Ptwo{\textrm{P}_{\textrm{II}}}

\newcommand\Pfive{\textrm{P}_{\textrm{V}}}

\title[Elliptic Asymptotic Behaviour of $q$-$\Ptwo$]{Elliptic asymptotic behaviour of $q$-Painlev\'e transcendents}
\author{Joshua Holroyd}
\date{\today}

\begin{document}
\begin{abstract}
    The discrete Painlev\'e equations have mathematical properties closely related to those of the differential Painlev\'e equations. We investigate the appearance of elliptic functions as limiting behaviours of $q$-Painlev\'e transcendents, analogous to the asymptotic theory of classical Painlev\'e transcendents. We focus on the $q$-difference second Painlev\'e equation ($q$-$\Ptwo$) in the asymptotic regime $|q-1|\ll1$, showing that generic leading-order behaviour is given in terms of elliptic functions and that the slow modulation in this behaviour is approximated in terms of complete elliptic integrals.
\end{abstract}
\maketitle

\section{Introduction} Given quantities $a,q,t_0\in\mathbb{C}^*$ with $q\neq1$,  we consider the $q$-difference second Painlev\'e equation \cite{ramani1996discrete}
\begin{align*}
    q\text{-}\Ptwo:\qquad f(qt)f(t/q)=\frac{a(1+f(t)t)}{f(t)(f(t)+t)},
\end{align*}
where $f$ is a function of $t=t_n=t_0q^{n}$, $n\in\mathbb{Z}$; this becomes the second Painlev\'e equation in a continuum limit \cite{ramani1996discrete}. The equation $q$-$\Ptwo$ belongs to a class of integrable, second-order, nonlinear difference equations known as discrete Painlev\'e equations, which each tend to a classical Painlev\'e equation in a continuum limit. A multiplicative difference equation (termed `$q$-difference') is characterised by iteration of the independent variable $t$ along a $q$-spiral $t_0q^{\mathbb{Z}}\subset\mathbb{C}$ for some initial point $t_0\neq0$ and constant $q\neq 0,1$.

Painlev\'e and colleagues discovered the six classical Painlev\'e equations in their effort to categorise all second-order, ordinary differential equations (ODEs) with solutions featuring no movable singularities other than poles; now known as the Painlev\'e property \cite{painleve1902equations,fuchs1905quelque,gambier1910equations}. There has been growing interest in these equations in modern times, owing to their role as mathematical models in various physical contexts. They are considered the nonlinear counterpart to classical special functions \cite{bornemann2010request}.

Recently, there has been significant interest in discrete versions of Painlev\'e equations. These first arose in orthogonal polynomial theory \cite{shohat1939differential}, and more recently in the study of a two-dimensional model of quantum gravity \cite{its1990isomonodromy}. It is now known that for each classical Painlev\'e equation, there are numerous integrable, second-order discrete versions; many of these were discovered via the singularity confinement property \cite{ramani1991discrete}, which is considered a discrete analogue of the aforementioned Painlev\'e property. Due to their prominence in various nonlinear models of mathematical physics, discrete Painlev\'e equations are regarded as defining new nonlinear special functions, as is the case with continuous Painlev\'e equations.

In general, solutions of the Painlev\'e equations are highly transcendental functions, known as Painlev\'e transcendents, and cannot be expressed in terms of known functions. This concept also applies to the discrete versions of Painlev\'e equations. It is then unsurprising that there has been considerable research regarding the asymptotic analysis of Painlev\'e equations; several key solutions of physical interest were initially identified via asymptotic analysis \cite{mccoy1977painleve,wu1976spin} as the independent variable approaches a fixed singularity.

Consider the Boutroux form of the second Painlev\'e equation ($\Ptwo$), which is more amenable to asymptotic analysis as the independent variable becomes large \cite{boutroux1913recherches}
\begin{align*}
    u_{zz}=2u^3-u-\frac{\alpha+u_z}{z}+\frac{u}{9z^2}.
\end{align*}
Here, $\alpha$ is a constant parameter, and the subscripts denote differentiation. The above equation becomes autonomous in the limit $|z|\to\infty$, and we thus find that the following quantity is constant to leading order
\begin{align*}
    E(u,u_z)=\frac{1}{2}\left(u_z^2-u^4+u^2\right).
\end{align*}
Considering constant $E$ above, the generic solution $u(z)$ is a (Jacobi) elliptic function.

In the case of the first and second Painlev\'e equations, these behaviours were first considered by Boutroux \cite{boutroux1913recherches,boutroux1914recherches}. Furthermore, for $\Pone$ and $\Ptwo$, Joshi and Kruskal studied the modulation of these leading-order elliptic functions (i.e. the slow variation of $E$ above) as the angle of approach to infinity is varied within a bounded sector of the complex $z$-plane \cite{joshi1987connection,joshi1988asymptotic,joshi1992painleve}. An averaging method showed that the slow modulation of $E$ is given to leading order by complete elliptic integrals associated with the leading-order elliptic function. Analogous results for the third, fourth and fifth Painlev\'e equations are given in \cite{joshi2018asymptotic}.

In this article, we investigate these concepts in the $q$-difference setting. We consider $q$-$\Ptwo$ with $q=1+\epsilon$ where $0<|\epsilon|\ll1$ is a small parameter; that is, by decreasing the step size $|t_{n+1}-t_n|=|\epsilon t_n|$, we consider a generic solution of $q$-$\Ptwo$ while staying within a bounded region near an arbitrary point $t_0\in\mathbb{C}^*$. In this asymptotic context, we show that the leading-order autonomous form of $q$-$\Ptwo$ is solved in terms of Jacobi elliptic functions. Furthermore, we extend the averaging method and describe the slow evolution of this leading-order elliptic behaviour in terms of complete elliptic integrals. Finally, we consider critical points in the initial value space where solutions degenerate to being singly periodic, as is the case for continuous Painlev\'e transcendents along certain rays in the complex plane (Stokes boundaries).

This investigation provides a new fundamental link between discrete and continuous Painlev\'e equations and develops a new understanding of the generic behaviour of a $q$-difference Painlev\'e transcendent. Our main results are given by Theorems \ref{thrm:elliptic} and \ref{thrm:main2}.

\section{Elliptic Behaviour}
In this section, we consider the autonomous version of $q$-$\Ptwo$, obtained at leading order by setting $q=1+\epsilon$, where $\epsilon\in\mathbb{C}^*$ is a small parameter. After providing a bound on the error associated with the leading-order equation, we give the generic solution in terms of Jacobi elliptic functions in the limit $\epsilon\to0$.

We begin by formally defining a solution of $q$-$\Ptwo$, the quantity $E$ which is conserved at leading order, and a quantity $L$ which arises naturally when considering the slight difference in $E$ as we iterate $n\to n+1$.
\begin{definition}\label{def:f}
    Let $a,f_0,f_1,t_0,\epsilon\in\mathbb{C}^*$, with $|\epsilon|<1$. We define the sequence $t:\mathbb{N}\rightarrow\mathbb{C}$ by $t_n=t_0(1+\epsilon)^n$, and the corresponding sequence $f:\mathbb{N}\rightarrow\mathbb{C}$ by
    \begin{align}\label{eq:qP2}
        f_{n+1}f_{n-1}=\frac{a(1+f_nt_n)}{f_n(f_n+t_n)}.
    \end{align}
\end{definition}
\begin{definition}\label{def:E}
    We define the sequences $E:\mathbb{N}\rightarrow\mathbb{C}$ and $L:\mathbb{N}\rightarrow\mathbb{C}$ by
    \begin{align}\label{eq:invariant}
        E_n=&\frac{f_{n+1}^2f_n^2+t_0f_{n+1}f_n(f_{n+1}+f_n)+at_0(f_{n+1}+f_{n})+a}{f_{n+1}f_n},\\
        L_n=&\frac{(a+f_{n+1}f_n)(f_{n+1}+f_n)}{f_{n+1}f_n}.
    \end{align}    
\end{definition}
We study the sequence $f_n$ in a domain where it remains bounded and sufficiently away from zero, which is always possible due to the meromorphic nature of solutions of discrete Painlev\'e equations, leading to the following definition.
\begin{definition}\label{def:bounds}
    Given some $n\in\mathbb{N}$, we denote finite value $J_n\in\mathbb{R}^+$ such that for all $k\in\mathbb{N}$, with $0\leq k\leq n$, the sequence $f_k$ satisfies
    \begin{align*}
        \left|L_k\right|=\left|\frac{a}{f_{k+1}}+\frac{a}{f_k}+f_{k+1}+f_k\right|\leq J_n.
    \end{align*}
\end{definition}
Across an arbitrary $n\in\mathbb{N}$ number of steps, we proceed to give a bound on the difference $|E_n-E_{0}|$, and thus provide sufficient conditions for this difference to remain small as $\epsilon\to0$.
\begin{lemma}\label{lem:LOEq}
    Let $n\in\mathbb{N}$ and assume the existence of finite value $J_n$ as described in Definition \ref{def:bounds}. Then we obtain
    \begin{align*}
        |E_n-E_0|<2\left|t_0\right|J_n\left(e^{|\epsilon|n}-1\right),
    \end{align*}
    and thus $|E_n-E_{0}|\ll 1$ as $\epsilon\to0$ while $|\epsilon|n\ll1$.
\end{lemma}
\begin{proof}
    Let $n\in\mathbb{N}$. From definitions \ref{def:f} and \ref{def:E} we directly compute
    \begin{align}\label{eq:EdiffFirst}
        \nonumber E_n-E_{n-1}=&\frac{f_{n+1}-f_{n-1}}{f_{n+1}f_nf_{n-1}}\left(f_{n+1}f_n^2f_{n-1}+t_0f_{n+1}f_nf_{n-1}-at_0f_n-a\right)\\
        =&-(t_n-t_0)\frac{\left(f_{n+1}-f_{n-1}\right)\left(f_{n+1}f_{n-1}-a\right)}{f_{n+1}f_{n-1}}\nonumber\\
        =&-P_n(L_n-L_{n-1}),
    \end{align}
    where we denote $P_n$, for $n\in\mathbb{N}$, the binomial expansion
    \begin{align*}
        P_n=t_n-t_0=t_0\sum_{k=1}^{n}\begin{pmatrix}n\\k\end{pmatrix}\epsilon^k,\quad\text{with}\quad P_0=0.
    \end{align*}
    Then, it is straightforward to see that
    \begin{align*}
        \left|P_n\right|\leq&\left|t_0\right|\sum_{k=1}^{n}\frac{\left(n\left|\epsilon\right|\right)^k}{k!}<\left|t_0\right|\left(-1+\sum_{k=0}^{\infty}\frac{\left(n\left|\epsilon\right|\right)^k}{k!}\right)=\left|t_0\right|\left(e^{n|\epsilon|}-1\right),
    \end{align*}
    which immediately yields
    \begin{align}\label{ineq:Ediff}
        \left|E_n-E_{n-1}\right|<\left|t_0\right|\left|L_n-L_{n-1}\right|\left(e^{|\epsilon|n}-1\right).
    \end{align}
    
    Summing the equation $E_k-E_{k-1}=-P_k(L_k-L_{k-1})$ across $k=1$ to $k=n$ and applying summation by parts (analogous to integration by parts), we obtain
    \begin{align}\label{eq:diffE0n}
        E_n-E_0=-L_nP_n+t_0\epsilon\sum_{k=0}^{n-1}L_k\left(\frac{P_k}{t_0}+1\right).
    \end{align}
    Using the fact that $|P_k|<|t_0|(e^{k|\epsilon|}-1)$ for any $k\in\mathbb{N}$, it is straightforward to furthermore see that
    \begin{align*}
        \left|\sum_{k=0}^{n-1}L_k\left(\frac{P_k}{t_0}+1\right)\right|<\sum_{k=0}^{n-1}\left|L_k\right|e^{|\epsilon|k}\leq J_n\sum_{k=0}^{n-1}e^{|\epsilon|k}=J_n\frac{e^{|\epsilon|n}-1}{e^{|\epsilon|}-1}.
    \end{align*}
    Thus from Equation \eqref{eq:diffE0n} we obtain
    \begin{align*}
        |E_n-E_0|<&|L_nP_n|+|t_0\epsilon|J_n\frac{e^{|\epsilon|n}-1}{e^{|\epsilon|}-1}<2\left|t_0\right|J_n\left(e^{|\epsilon|n}-1\right),
    \end{align*}
    as required.
\end{proof}
Having established an error-bound, we consider the leading-order, autonomous equation
\begin{align}\label{eq:LOeq}
    f_{n+1}^2f_n^2+t_0f_{n+1}f_n(f_{n+1}+f_n)-E_0f_{n+1}f_n+at_0(f_{n+1}+f_{n})+a=H_n,
\end{align}
for arbitrary $n\in\mathbb{N}$, where by Lemma \ref{lem:LOEq}
\begin{align}\label{eq:HBound}
    |H_n|=|f_{n+1}f_n(E_n-E_0)|<2\left|t_0f_{n+1}f_n\right|J_n\left(e^{|\epsilon|n}-1\right),
\end{align}
so that $|H_n|\ll1$ as $\epsilon\to0$ while $|\epsilon|n\ll1$, given the existence of finite value $J_n\in\mathbb{R}^+$.

We proceed to define a linear fractional transformation of the function $f_n$, thus converting the LHS of Equation (\ref{eq:LOeq}) to the canonical bi-quadratic form $x^2y^2+\gamma(x^2+y^2)+\zeta xy+1$ for some $\gamma,\zeta\in\mathbb{C}$. These constants are provided in the following lemma, where we take the case $a=1$. Setting $a=1$ provides a more tractable illustrative example while not being a meaningful degeneration of the problem.
\begin{lemma}\label{lem:cononForm}
    With conditions $E_0\neq 2\pm4t_0$, define the sequence $F:\mathbb{N}\rightarrow\mathbb{C}$ by
    \begin{align*}
        f_n=\frac{c+F_n}{c-F_n}\quad\text{with $c\in\mathbb{C}^*$ such that}\quad c^4=\frac{2-E_0-4t_0}{2-E_0+4t_0}.
    \end{align*}
    Then considering bounded $n\in\mathbb{N}$ and taking the case $a=1$, Equation (\ref{eq:LOeq}) becomes
    \begin{align*}
        F_{n+1}^2F_{n}^2+\gamma\left(F_{n+1}^2+F_{n}^2\right)+\zeta F_{n+1}F_n+1=K_n\to0,\quad\epsilon\to0,
    \end{align*}
    where
    \begin{align*}
        \gamma=\frac{(2+E_0)c^2}{2-E_0-4t_0},\quad\zeta=\frac{8c^2}{2-E_0-4t_0},\quad K_n=\frac{(c-F_n)^2(c-F_{n+1})^2}{2-E_0-4t_0}H_n.
    \end{align*}
\end{lemma}
\begin{proof}
    The result may be verified straightforwardly via substitution. Considering Equation \eqref{eq:HBound}, we see that $K_n\to0$ as $\epsilon\to0$ given $n$ is bounded (or $|\epsilon|n\ll1$), $F_n$ is finite and $E_0\neq 2\pm4t_0$. The condition $E_0\neq 2\pm4t_0$ on the initial value space of $f$ is discussed in the subsequent remark.
\end{proof}
\begin{remark}
    The transformation described in Lemma \ref{lem:cononForm}, applied to Equation \eqref{eq:LOeq}, breaks down with initial conditions causing $E_0=2\pm4t_0$. By observing Equation \eqref{eq:qP2}, we see that these conditions correspond to stationary solutions $f_n=\pm1$ for all $n\in\mathbb{N}$ (in the $a=1$ case). Expansions of the solution $f$ near critical points in the initial value space are given in Appendix \ref{sec:CritPoints}.
\end{remark}
This section's core result is that $F_n$ described in Lemma \ref{lem:cononForm} is given by the Jacobi elliptic sine function in the limit $\epsilon\to0$.
\begin{theorem}\label{thrm:elliptic}
    Consider bounded $n\in\mathbb{N}$ and  $a=1$. Regarding the sequence $F_n$ described in Lemma \ref{lem:cononForm} we obtain
    \begin{align*}
        F_n\to A\,\text{sn}(z_0+pn;k),\quad\epsilon\rightarrow0.
    \end{align*}
    Here, sn$(z;k)$ is Jacobi's elliptic sine function of argument $z$ and modulus $k$. Furthermore, $A^2=k$ and $k$ satisfying $0<|k|<1$ is uniquely determined by
    \begin{align*}
        k^2+\beta k+1=0,\quad \beta:=\frac{4+4\gamma^2-\zeta^2}{4\gamma}.
    \end{align*}
    Finally, $p$ is chosen such that $k\text{sn}^2(p;k)=-1/\gamma$ and $z_0$ is representative of the initial condition $f_0$, satisfying $A\text{sn}(z_0;k)=F_0$.
\end{theorem}
\begin{proof}
    Let $A,k,p,z\in\mathbb{C}$ and define $G(z):=A\,\text{sn}(z;k)$. The result follows by applying the well-known addition theorem regarding the Jacobi sine function:
    \begin{align*}
        \text{sn}(x+y)=&\frac{\text{sn}(x)\text{cn}(y)\text{dn}(y)+\text{sn}(y)\text{cn}(x)\text{dn}(x)}{1-k^2\text{sn}^2(x)\text{sn}^2(y)},\quad\forall x,y\in\mathbb{C},
    \end{align*}
    where we have suppressed the modulus $k$ for simplicity, and the elliptic functions sn, cn and dn are related by
    \begin{align*}
        \text{cn}^2(x)+\text{sn}^2(x)=1,\quad\text{dn}^2(x)+k^2\text{sn}^2(x)=1,\quad \forall x\in\mathbb{C}.
    \end{align*}
    Applying the above formulas, we find that $G(z)$ satisfies
    \begin{align*}
        G(z)^2G(z+p)^2+\gamma\left(G(z)^2+G(z+p)^2\right)+\zeta G(z)G(z+p)+1=0,\quad\forall z\in\mathbb{C},
    \end{align*}
    given we choose $A$ satisfying $A^2=k$, $p$ satisfying $G(p)^2=-1/\gamma$, and $k$ satisfying
    \begin{align*}
        k^2+\beta k+1=0,\quad \beta:=\frac{4+4\gamma^2-\zeta^2}{4\gamma}.
    \end{align*}
    Without loss of generality, we take the solution of the above quadratic satisfying $|k|<1$, noting that this equation is invariant under $k\to1/k$. Considering Lemma \ref{lem:cononForm}, it now follows that $F_n\to G(z_0+pn)$ as $\epsilon\to0$, given we choose appropriate initial condition $z_0\in\mathbb{C}$ satisfying $G(z_0)=F_0$.
\end{proof}
Utilising Theorem \ref{thrm:elliptic}, we proceed to define appropriate continuous functions $f(x),E(x),L(x)$, which will be important to Section \ref{sec:averaging}.
\begin{definition}\label{def:cont}
    Following Lemma \ref{lem:cononForm} and Theorem \ref{thrm:elliptic}, we define the continuous function $f:\mathbb{R}\rightarrow\mathbb{C}$ so that $f(n)=f_n$ for bounded $n\in\mathbb{N}$, and
    \begin{align*}
        f(x)\to\frac{c+A\,\text{sn}(z_0+px)}{c-A\,\text{sn}(z_0+px)},\quad\epsilon\rightarrow0,
    \end{align*}
    for bounded $x\in\mathbb{R}$. We likewise define corresponding continuous functions $E(x)$ and $L(x)$ (see Definition \ref{def:E}).
\end{definition}
In Figures \ref{fig:numeric} and \ref{fig:numeric2}, we show numerically computed sequences $f_n$ satisfying Equation (\ref{eq:qP2}), and the corresponding approximations given by Theorem \ref{thrm:elliptic}.
\begin{figure}
    \centering
    \begingroup
    \begin{tabular}{c}
    \includegraphics[width=0.9\textwidth]{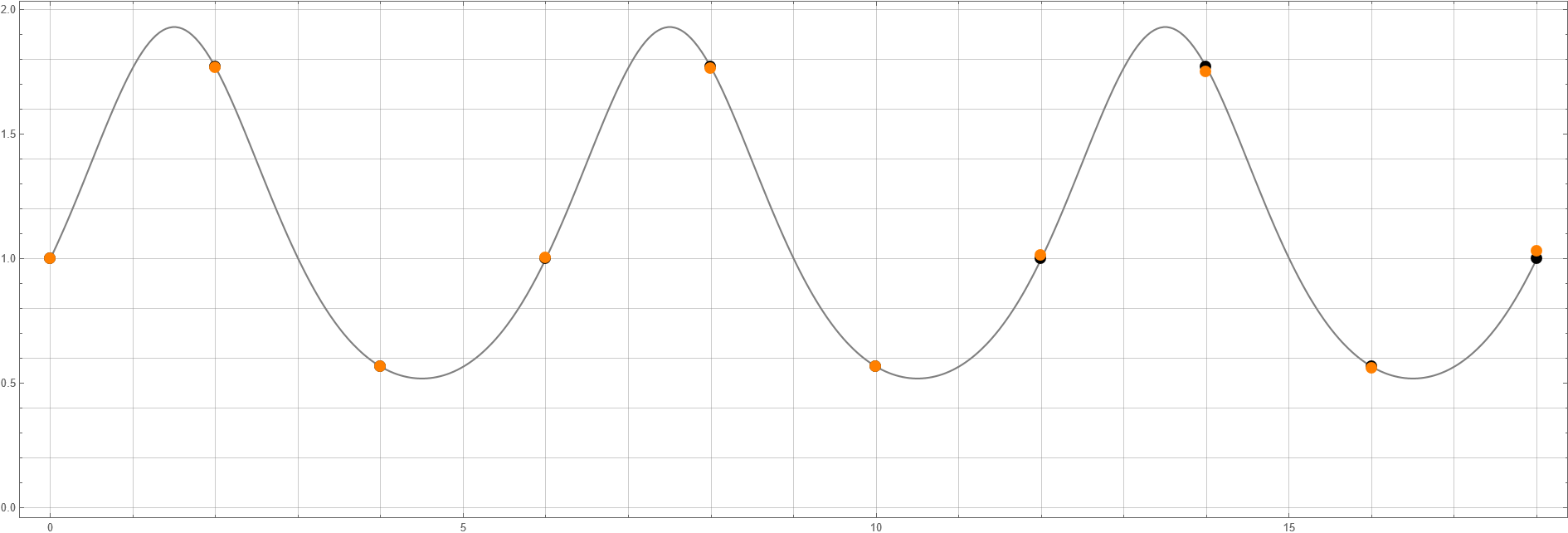} \\
    \includegraphics[width=0.9\textwidth]{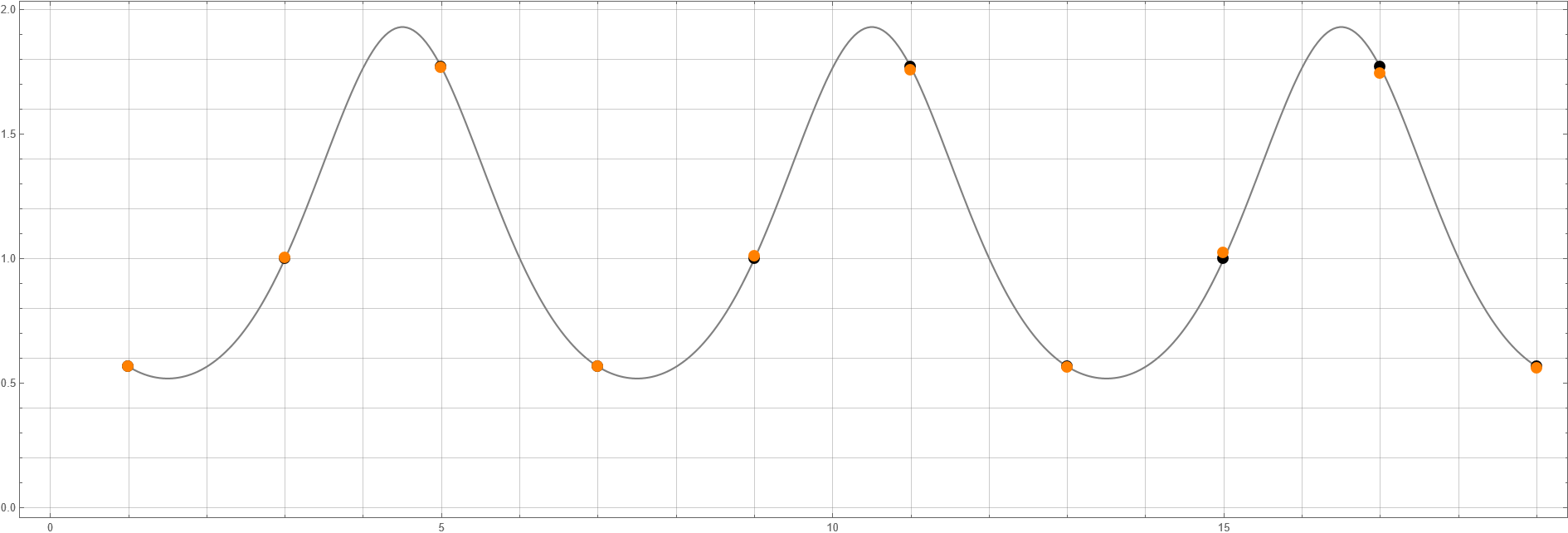}
    \end{tabular}
    \endgroup
    \caption{Elliptic behaviour example A} Parameters and initial conditions are $a=f_0=t_0=1$, $E_0=20/3$ and $\epsilon=1/1000$. The top and bottom plots distinguish between even and odd iterates of $f_n$. The orange points represent the numerically computed sequence $f_n$ solving Equation (\ref{eq:qP2}), and the black points are given by Theorem \ref{thrm:elliptic}.
    \label{fig:numeric}
\end{figure}
\begin{figure}
    \centering
    \begingroup
    \begin{tabular}{c}
    \includegraphics[width=0.9\textwidth]{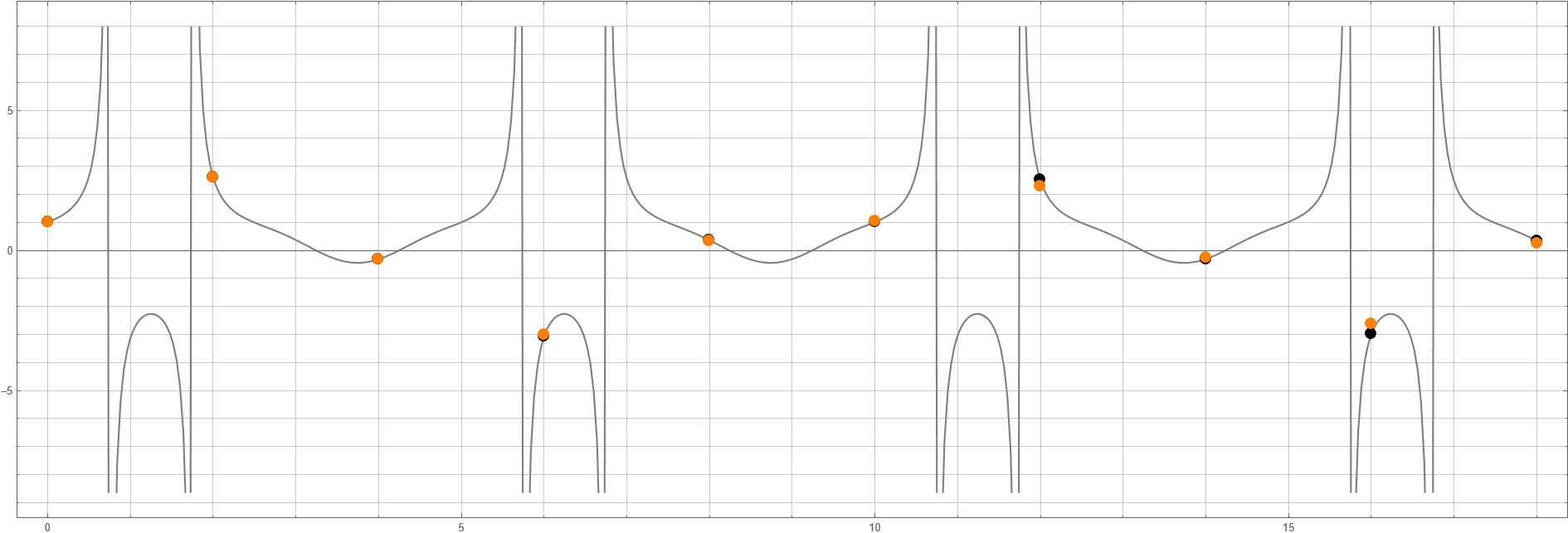} \\
    \includegraphics[width=0.9\textwidth]{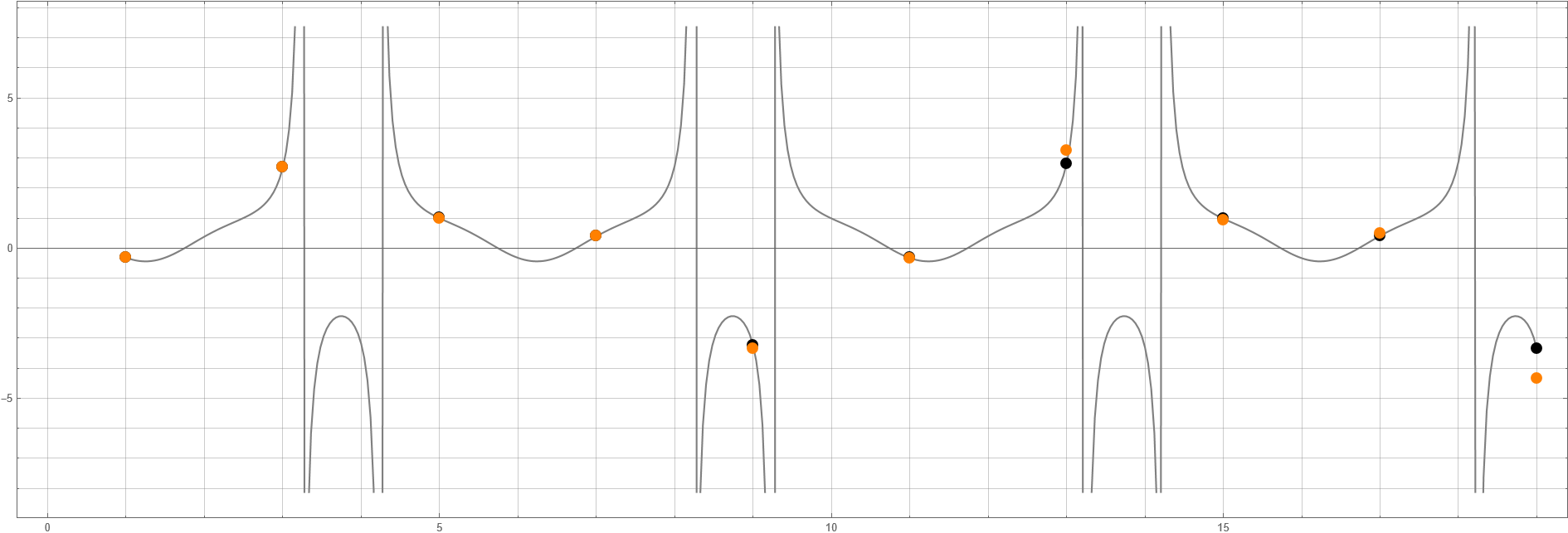}
    \end{tabular}
    \endgroup
    \caption{Elliptic behaviour example B} Parameters and initial conditions are $a=f_0=1$, $t_0=-\sqrt{2}$, $E_0=-1.4$ and $\epsilon=1/1000$. The top and bottom plots distinguish between even and odd iterates of $f_n$. The orange points represent the numerically computed sequence $f_n$ solving Equation (\ref{eq:qP2}), and the black points are given by Theorem \ref{thrm:elliptic}.
    \label{fig:numeric2}
\end{figure}

We proceed to discuss the periodicity of $G(z):=A\,\text{sn}(z;k)$, this having two fundamental periods which are linearly independent in the complex $z$-plane, given $k\neq0,\pm1$.
\begin{definition}\label{def:periods}
    Given the function $G(z):=A\,\text{sn}(z;k)$, where $A^2=k$ and $0<|k|<1$, we define constants $\Omega_1,\Omega_2\in\mathbb{C}^*$ such that $G(z+\Omega_j)=G(z)$ for all $z\in\mathbb{C}$ and $j\in\{1,2\}$. Furthermore, $\Omega_1\neq r\Omega_2$ for all $r\in\mathbb{R}^*$.
\end{definition}
\begin{corollary}\label{cor:Periods}
     The constants $\Omega_1,\Omega_2$ described in Definition \ref{def:periods} are given by the complete elliptic integral of the first kind
        \begin{align}\label{int:Omega}
        \Omega(k):=\oint\frac{\text{d}u}{\sqrt{(1-u^2)(1-k^2u^2)}}.
    \end{align}
    These may be expanded in powers of $m:=k^2$ as follows
    \begin{align*}
        \Omega_1(m)=2h(m),\quad\Omega_2(m)=ih(1-m),\quad h(m):=&\sum_{j=0}^{\infty}\frac{\Gamma(j+1/2)^2}{\Gamma(j+1)^2}m^j.
    \end{align*}
\end{corollary}
\begin{proof}
    Let $G(z):=A\,\text{sn}(z;k)$, where $A^2=k$ and $0<|k|<1$. For all $z\in\mathbb{C}$, it is well known that the function $G(z)$ satisfies
    \begin{align*}
        \left(\frac{dG}{dz}\right)^2=\left(k-G^2\right)\left(1-kG^2\right).
    \end{align*}
    The two independent periods of $G(z)$ correspond to the two independent closed contours associated with the elliptic integral of the first kind
    \begin{align*}
        \Omega(k):=\oint\frac{\text{d}G}{\sqrt{(k-G^2)(1-kG^2)}}=\oint\frac{\text{d}u}{\sqrt{(1-u^2)(1-k^2u^2)}}.
    \end{align*}
    Letting prime denote differentiation with respect to $k$, we find that the integral $\Omega$ satisfies the ODE
    \begin{align*}
        k(1-k^2)\Omega''+(1-3k^2)\Omega'-k\Omega=0.
    \end{align*}
    Converting to the parameter $m=k^2$ (and prime now denotes differentiation with respect to $m$) gives the canonical hypergeometric differential equation
    \begin{align*}
        m(1-m)\Omega''+(1-2m)\Omega'-\frac{1}{4}\Omega=0.
    \end{align*}
    This ODE has regular singular points at $m\in\{0,1\}$. Furthermore, it is invariant under the transformation $m\rightarrow1-m$. Linearly independent solutions are given by
    \begin{align*}
        h_1(m)=&\sum_{j=0}^{\infty}\frac{\Gamma(j+1/2)^2}{\Gamma(j+1)^2}m^j,\\
        h_2(m)=&\log(m)\sum_{j=0}^{\infty}\frac{\Gamma(j+1/2)^2}{\Gamma(j+1)^2}m^j\\
        &+2\sum_{j=1}^{\infty}\frac{\Gamma(j+1/2)^2}{\Gamma(j+1)^2}\left(\psi(j+1/2)-\psi(j+1)+2\log(2)\right)m^j,
    \end{align*}
    where we denote the gamma function $\Gamma(z)$ and the digamma function $\psi(z):=\Gamma'(z)/\Gamma(z)$. Expansions around $m=1$ are obtained by simply transforming $m\rightarrow1-m$.
\end{proof}
\section{Averaging Method}\label{sec:averaging}
This section considers the order $\epsilon$ perturbation associated with the leading-order autonomous equation. We approximate the variation in $E_n$ over a fundamental period of the leading-order elliptic behaviour in terms of complete elliptic integrals.

We first give an expression for the $\mathcal{O}(\epsilon)$ part of $E_n-E_{0}$ for arbitrary $n$, and provide a bound on smaller order terms which we denote by $M_n$.
\begin{lemma}\label{lem:diffEn0}
    For any $n\in\mathbb{N}$ we have
    \begin{align}\label{eq:difference}
        E_n-E_{0}=&-\left(nL_n-\sum_{k=0}^{n-1}L_k\right)t_0\epsilon+M_n,
    \end{align}
    where
    \begin{align*}
        |M_n|<2|t_0|J_n\left(e^{|\epsilon|n}-1-|\epsilon|n\right)=\mathcal{O}\left(|\epsilon|^2n^2\right),\quad|\epsilon|n\ll1.
    \end{align*}
\end{lemma}
\begin{proof}
    Defining $\Tilde{P}_n=P_n-t_0\epsilon n$, following Equation \eqref{eq:EdiffFirst} we have
    \begin{align*}
        E_k-E_{k-1}=&-t_0\epsilon(L_k-L_{k-1})k-(L_k-L_{k-1})\Tilde{P}_k,\quad \forall k\in\mathbb{N}.
    \end{align*}
    Therefore, for arbitrary $n\in\mathbb{N}$, we sum the above equations to obtain
    \begin{align*}
        E_n-E_{0}=&-\left(nL_n-\sum_{k=0}^{n-1}L_k\right)t_0\epsilon+M_n,
    \end{align*}
    where, similar to Equation \eqref{eq:diffE0n}, we express $M$ in the form
    \begin{align*}
        M_n=&-\sum_{k=1}^{n}(L_k-L_{k-1})\Tilde{P}_k=-\Tilde{P}_nL_n+t_0\epsilon\sum_{k=0}^{n-1}L_k\left(\frac{\Tilde{P}_k}{t_0}+\epsilon k\right).
    \end{align*}
    Then using the fact that $|\Tilde{P}_n|<|t_0|(e^{|\epsilon|n}-1-|\epsilon|n)$ for all $n\in\mathbb{N}$, we see that
    \begin{align*}
        |M_n|<&|t_0|J_n\left\{\left(e^{|\epsilon|n}-1-|\epsilon|n\right)+|\epsilon|\sum_{k=0}^{n-1}\left(e^{|\epsilon|k}-1\right)\right\}\\
        <&2|t_0|J_n\left(e^{|\epsilon|n}-1-|\epsilon|n\right),
    \end{align*}
    as required.
\end{proof}
Now, we apply the averaging method, deducing an approximation for the variation in $E$ over a fundamental period of the leading-order elliptic behaviour. For $a=1$ and all $0\leq n\ll1/|\epsilon|$, as $\epsilon\to0$ we have coordinates $(f_{n+1},f_n)$ confined to a level curve of the surface
\begin{align*}
    E_0(x,y)=\frac{x^2y^2+t_0xy(x+y)+t_0(x+y)+1}{xy}.
\end{align*}
In what proceeds, we suppose that there is some $\eta\in\mathbb{N}$ which is a \emph{near-period} of the discrete function $f_n$. That is, $\eta$ is close to a finite period of the leading order elliptic behaviour, given by Integral \eqref{int:Omega} with some non-trivial closed contour.
\begin{theorem}\label{thrm:main2}
    Let $a=1$ and assume $|L^{(k)}(0)|\leq R$ for all $k\in\mathbb{N}$ and some finite $R\in\mathbb{R}^+$. Furthermore, suppose we have $\eta\in\mathbb{N}$ such that $|\eta-\Omega/p|\ll1$, where $\Omega(k)$ is a period of the leading-order elliptic function given by Integral \eqref{int:Omega}. Then, regarding the difference $E_\eta-E_0$ we obtain
    \begin{align*}
        E_\eta-E_{0}\sim&-\frac{t_0\epsilon}{p}\left\{(4+L_0)\Omega(k)-8\Pi(k/c^2,k)\right\}+S_\eta,\quad\epsilon\to0,
    \end{align*}
    where $\Pi(\alpha^2,k)$ is the complete elliptic integral of the third kind with modulus $k$ and parameter $\alpha^2$, and
    \begin{align*}
        |S_\eta|<R|t_0\epsilon|\left\{|\eta-\Omega/p|e^{|\eta-\Omega/p|}+\left(\left|\Omega/p\right|+1\right)\left(e^{|\eta-\Omega/p|}-1\right)\right\}=\mathcal{O}\left(|\epsilon||\eta-\Omega/p|\right).
    \end{align*}
\end{theorem}
\begin{proof}
    For finite $\eta\in\mathbb{N}$, in Lemma \ref{lem:diffEn0} we showed that
    \begin{align*}
        E_\eta-E_{0}\sim&-\left((\eta+1)L_\eta-L_0-\sum_{k=1}^{\eta}L_k\right)t_0\epsilon,\quad\epsilon\to0.
    \end{align*}
    We apply the Euler-Maclaurin formula to the sum in the above equation, seeing that
    \begin{align*}
        \sum_{k=1}^{\eta}L(k)=\int_{0}^{\Omega/p}L(x)dx+\int_{\Omega/p}^{\eta}L(x)dx+\sum_{k=1}^{\infty}\frac{B_k}{k!}\left(L^{(k-1)}(\eta)-L^{(k-1)}(0)\right),
    \end{align*}
    where $B_k$, for $k\in\mathbb{N}$, are the Bernoulli numbers. Now, we utilise the expansion
    \begin{align*}
        L^{(k)}(\eta)=&\sum_{j=k}^{\infty}\frac{L^{(j)}(\Omega/p)}{(j-k)!}(\eta-\Omega/p)^{j-k},\quad k\in\mathbb{N},
    \end{align*}
    also noting that $L^{(j)}(\Omega/p)\to L^{(j)}(0)$ as $\epsilon\to0$ for all $j\in\mathbb{N}$. Applying these concepts, we obtain
    \begin{align}\label{eq:rawInt}
        &E_\eta-E_{0}\sim-t_0\epsilon\left(\frac{\Omega}{p}L_0-\int_{0}^{\Omega/p}L(x)dx\right)+S_\eta,
    \end{align}
    where
    \begin{align*}
        S_\eta=&-t_0\epsilon\sum_{k=1}^{\infty}\frac{\theta_k}{k!}(\eta-\Omega/p)^{k},
    \end{align*}
    and
    \begin{align*}
        \theta_k=(k-1)L^{(k-1)}(0)+\left(\Omega/p+1\right)L^{(k)}(0)-\sum_{j=1}^{\infty}\frac{B_j}{j!}L^{(j+k-1)}(0).
    \end{align*}
    Using the assumption that $|L^{(k)}(0)|\leq R$ for all $k\in\mathbb{N}$, and also noting
    \begin{align*}
        \sum_{j=1}^{\infty}|B_j/j!|<1,
    \end{align*}
    we may verify that
    \begin{align*}
        |S_\eta|<R|t_0\epsilon|\left\{|\eta-\Omega/p|e^{|\eta-\Omega/p|}+\left(\left|\Omega/p\right|+1\right)\left(e^{|\eta-\Omega/p|}-1\right)\right\}\ll|\epsilon|,
    \end{align*}
    for $|\eta-\Omega/p|\ll1$. Regarding the integral in Equation \eqref{eq:rawInt}, recall that in the $a=1$ case we have simply
    \begin{align*}
        L(x)=f(x)+\frac{1}{f(x)}+f(x+1)+\frac{1}{f(x+1)},
    \end{align*}
    and the continuous behaviour of $f(x)$ as $\epsilon\to0$ is stated in Definition \ref{def:cont}. Then, since the integral is across a period of the leading-order elliptic behaviour, we find that
    \begin{align*}
        \int_{0}^{\Omega/p}L(x)dx=&\frac{4}{p}\oint\left(\frac{1+ku^2/c^2}{1-ku^2/c^2}\right)\frac{\text{d}u}{\sqrt{(1-u^2)(1-k^2u^2)}}\\
        =&\frac{4}{p}\left(2\Pi(k/c^2,k)-\Omega(k)\right),
    \end{align*}
    where $\Pi(k/c^2,k)$ is a standard notation for the complete elliptic integral of the third kind, with parameter $k/c^2$ and modulus $k$. Applying this to Equation \eqref{eq:rawInt} completes the proof.
\end{proof}
\section{Conclusion}
In this article, we gave a generic transcendental solution $f(t)$ of the $q$-difference second Painlev\'e equation to leading order by elliptic functions in a bounded region near an arbitrary point in the $t$-plane, achieved by setting $t=t_0(1+\epsilon)^n$ with a small parameter $|\epsilon|\ll1$. Furthermore, by the averaging method, we studied the perturbation of this leading-order behaviour as we traverse an edge of the associated period parallelogram; this slow modulation is given in terms of complete elliptic integrals.

Future questions include how this analysis applies to other $q$-difference Painlev\'e equations, particularly considering that analogous results between $\Pone$ to $\Pfive$ share a remarkable similarity \cite{joshi2018asymptotic}. Questions remain concerning elliptic-type behaviours associated with more complex $q$-difference Painlev\'e equations (potentially in a different asymptotic limit to what has been considered here) and how this reflects the asymptotic theory of classical Painlev\'e equations.
\appendix
\section{Critical Points}\label{sec:CritPoints} In the $a=1$ case, we find that critical points $k=\pm1$ are given by two potential conditions in terms of $E_0$: (a) $2+E_0+t_0^2=0$ and (b) $|E_0|\rightarrow\infty$. On the other hand, we find that the three conditions give $k=0$: (a) $2+E_0=0$, (b) $2-E_0+4t_0=0$ and (c) $2-E_0-4t_0=0$. We proceed to give a summary of expansions for $f_n$ as $E_0$ approaches each of these critical points, whereby $f$ is now degenerating from elliptic to constant or singly-periodic in the limit $\epsilon\to0$. Here we may apply known results regarding the elliptic sine sn$(z;k)$ as the modulus $k$ approaches $0,\pm1$, recalling that for bounded $n\in\mathbb{N}$ we have
\begin{align*}
    f_n\sim g(z):=\frac{c+A\,\text{sn}(z;k)}{c-A\,\text{sn}(z;k)},\quad z:=z_n=z_0+pn,\quad\epsilon\rightarrow0.
\end{align*}
These degenerate behaviours correspond to certain algebraic curves in the initial value space, those being level curves of the surface given by
\begin{align*}
    E_0(f_0,f_1)=\frac{f_{1}^2f_0^2+t_0f_{1}f_0(f_{1}+f_0)+t_0(f_{1}+f_{0})+1}{f_{1}f_0},
\end{align*}
where the specific level curves in question include (or are solely) a critical point of the surface.
\subsection{Case $k^2\to1$ (a)} For $\delta:=E_0+2+t_0^2$, behaviour as $\delta\to0$ is given by
\begin{align*}
    g(z)=\frac{c_0+\tanh(z)}{c_0-\tanh(z)}+\mathcal{O}\left(\delta^{1/2}\right),\quad c_0^4=\frac{(t_0-2)^2}{(t_0+2)^2}.
\end{align*}
\subsection{Case $k^2\to1$ (b)} As $|E_0|\to\infty$ we obtain
\begin{align*}
    g(z)=\frac{1+\tanh(z)}{1-\tanh(z)}-\frac{4t_0\tanh(z)}{(\tanh(z)-1)^2}E_0^{-1}+\mathcal{O}\left(E_0^{-3/2}\right).
\end{align*}
\subsection{Case $k^2\to0$ (a)}\label{subsec:rational1} For $\delta:=E_0+2$, behaviour as $\delta\to0$ is given by
\begin{align*}
    &g(z)=1+\frac{2A_0\sin(z)}{c_0}\delta^{1/2}+\frac{2A_0^2\sin^2(z)}{c_0^2}\delta+\mathcal{O}\left(\delta^{3/2}\right),\\
    &c_0^4=\frac{1-t_0}{1+t_0},\quad A_0^4=-\frac{(t_0+1)(t_0-1)}{16t_0^4}.
\end{align*}
\subsection{Case $k^2\to0$ (b)}\label{subsec:rational} For $\delta:=E_0-2(1+2t_0)$, behaviour as $\delta\to0$ is given by
\begin{align*}
    &g(z)=1+\frac{2A_0\sin(z)}{c_0}\delta^{1/2}+\frac{2A_0^2\sin^2(z)}{c_0^2}\delta+\mathcal{O}\left(\delta^{3/2}\right),\\
    &c_0^4=8t_0,\quad A_0^4=\frac{(t_0+1)^2}{2t_0(t_0+2)^2}.
\end{align*}
\subsection{Case $k^2\to0$ (c)} For $\delta:=E_0-2(1-2t_0)$, behaviour as $\delta\to0$ is given by
\begin{align*}
    g(z)=\frac{c_0+A_0\sin(z)}{c_0-A_0\sin(z)}+\mathcal{O}(\delta),\quad c_0^4&=-\frac{1}{8t_0},\quad A_0^4=-\frac{(t_0-1)^2}{2t_0(t_0-2)^2}.
\end{align*}
\begin{remark}
    We note that Cases \ref{subsec:rational1} and \ref{subsec:rational} corresponds to the exact solution $f(t)\equiv1$ when $a=1$, which is the seed of a known hierarchy of rational solutions which exist for parameter values $a=q^{4k}$, $k\in\{0,1,2,\dots\}$.
\end{remark}
\printbibliography
\end{document}